\documentclass{BoKaVe20}

\title[The generating function of Kreweras walks with interacting boundaries is not algebraic]{The generating function of Kreweras walks with interacting boundaries is not algebraic}

\author[A. Bostan, M. Kauers, T. Verron]{Alin Bostan\thanks{\href{mailto:alin.bostan@inria.fr}{alin.bostan@inria.fr}; supported by the French ANR grant \href{https://specfun.inria.fr/chyzak/DeRerumNatura/}{DeRerumNatura},  ANR-19-CE40-0018.}\addressmark{1}, Manuel Kauers\thanks{\href{mailto:manuel@kauers.de}{manuel.kauers@jku.at}; supported by the Austrian FWF grants F5004 and P31571-N32.}\addressmark{2}   
\and  Thibaut Verron\thanks{\href{mailto:thibaut.verron@gmail.com}{thibaut.verron@jku.at}; supported by the Austrian FWF grants F5004 and P31571-N32.}\addressmark{2}
}   

\address{\addressmark{1}Inria, Universit{\'e}  Paris-Saclay, 1 rue Honor{\'e} d'Estienne d'Orves, 91120 Palaiseau, France \\ \addressmark{2}Institute for Algebra, Johannes Kepler University, Linz, Austria}

\abstract{Beaton, Owczarek and Xu (2019) studied generating functions of
Kreweras walks and of reverse Kreweras walks in the quarter plane, with
interacting boundaries. They proved that for the reverse Kreweras step set,
the generating function is always algebraic, and for the Kreweras step set,
the generating function is always D-finite. However, apart from the particular
case where the interactions are symmetric in $x$ and~$y$, they left open the
question of whether the latter one is algebraic. Using computer algebra tools,
we confirm their intuition that the generating function of Kreweras walks is
not algebraic, apart from the particular case already identified.}

\resume{Beaton, Owczarek et Xu (2019) ont étudié les séries génératrices des
marches avec interactions de type Kreweras et Kreweras inversé dans le quart
de plan. Pour le modèle de Kreweras renversé, ils ont prouvé que la série
génératrice est toujours algébrique, et pour le modèle de Kreweras, que la
série génératrice est toujours D-finie. Cependant, mis à part le cas
particulier où les interactions sont symétriques en $x$ et $y$, ils ont laissé
ouverte la question de savoir si cette dernière est algébrique. En utilisant
du calcul formel, nous confirmons leur intuition que la série génératrice des
marches de Kreweras n'est jamais algébrique, mis à part le cas particulier
déjà identifié.}


\begin{document}

\maketitle

\section{Introduction} \label{sec:introduction}

It is always interesting to know whether a generating function is D-finite,
because D-finiteness gives easy access to a lot of useful information about
the series. It is also interesting to know whether a D-finite series is
algebraic, because algebraicity gives access to even more useful information
or makes more efficient algorithms applicable. Every algebraic series is
D-finite but not vice versa, and it is a notoriously difficult problem to
decide for a given D-finite power series whether it is algebraic or
not~\cite[\S4(g)]{Stanley80}. There exists an algorithm for deciding whether a
given linear differential equation has only algebraic
solutions~\cite{Singer80}. This algorithm can be generalized in order to
compute, for a given linear differential operator~$L$, another differential
operator $L^{\text{alg}}$, whose solution space is spanned by the algebraic
solutions of $L$~\cite{Singer14}. The operator $L^{\text{alg}}$ can then be
used to decide whether a specific solution~$y$ of $L(y)=0$ is algebraic or
transcendental. However, the algorithm for computing $L^{\text{alg}}$ is very
expensive, and to our knowledge it was never implemented.

A popular and simple check is to inspect the asymptotic behaviour of the
coefficient sequence: if it is not of the form $c\phi^n n^\alpha$ with
$\alpha\in\set Q\setminus\set Z_{<0}$, then the series is
transcendental~\cite{Flajolet87}. However, this condition is only necessary
but not sufficient. The purpose of this paper is to highlight a less popular
condition which is also necessary but not sufficient, and which can be tried
in many cases where the asymptotic test fails. The method consists in finding
(using computer algebra) a closed form representation of the D-finite function
in question and to prove (also using computer algebra) that the function has a
logarithmic singularity.

{Our method is an illustration of the \emph{guess-and-prove} paradigm, which
is classically used to prove algebraicity~\cite{BoKa10}: one guesses an
algebraic equation, then post-certifies it. Transcendence is a more difficult
task, as one needs to prove that no algebraic equation exists. However, if one
can still guess a differential equation, and solve it in explicit form, then
the explicit solution can lead to transcendence proofs. This is the
methodology promoted here. In order to facilitate its application to other
examples, we include a detailed description of the required computer algebra
calculations for Maple and Mathematica\footnote{Also available online:
\url{http://www.algebra.uni-linz.ac.at/research/kreweras-interacting}}. }

As a concrete example, we consider a power series that appears in a recent
study of restricted lattice walk models with interacting boundaries. A model
is determined by a step set $S\subseteq\{-1,0,1\}^2\setminus\{(0,0)\}$ and
consists of walks in the quarter plane $\set N^2$ starting at~$(0,0)$. For
each step set $S$, Beaton et al. \cite{Beaton_2019,Beaton_2020} are interested
in the generating functions $Q(a,b,c,x,y,t)\in\set Q[[a,b,c,x,y,t]]$, where
$[a^hb^vc^ux^i y^j t^n]Q$ is the number of walks of length $n$ starting at
$(0,0)$ ending at $(i,j)$, with $h$ visits of the horizontal axis, $v$~visits
of the vertical axis, and $u$~visits of the origin. Among other things, they
show that the generating function is algebraic for the step set
$\{\step{-1,-1},\step{1,0},\step{0,1}\}$ (known as reverse Kreweras), and that
the generating function is D-finite for the step set
$\{\step{1,1},\step{-1,0},\step{0,-1}\}$ (known as Kreweras). They conjecture
that this latter series is not algebraic, and this is what we will prove here.
More precisely, we will show the following:

\begin{thm}
  \label{prop:main-thm}
  Let $Q(a,b,c;x,y;t) \in \QQ[a,b,c][[x,y,t]]$ be the generating function counting Kreweras walks with interacting boundaries, restricted to the quarter plane.
  The generating function $Q(a,b,c;x,y;t)$ is not algebraic over $\QQ(a,b,c,x,y,t)$.

  Furthermore, for values $a,b,c \in \QQ$ with $c \neq 0$, the generating function $Q(a,b,c;x,y;t) \in \QQ[[x,y,t]]$ is algebraic over $\QQ(x,y,t)$ if and only if $a=b$.
\end{thm}

\section{Notations and the kernel equation}
\label{sec:notations}

We first recall some notations used in this paper.
Whenever possible, we follow the notations used in \cite{Beaton_2020}.
Let $R$ be an integral domain with fraction field $K$.
We denote:
\begin{itemize}
  \item $R[t]$ the ring of polynomials in $t$ with coefficients in $R$;
  \item $K(t)$ the field of rational functions in $t$ with coefficients in $K$, which is the fraction field of $R[t]$; 
  \item $R[t,1/t]$ the ring of Laurent polynomials in $t$ with coefficients in $R$;
  \item $R[[t]]$ the ring of formal power series in $t$ with coefficients in $R$.
\end{itemize}

Given $f(t) \in K((t))$, we denote by $[t^{n}] f$ the coefficient of $t^n$ in $f(t)$, so that $f(t) = \sum_{n \in \ZZ} ([t^{n}]f) t^n$.
We denote by $[t^{>}]f$ the sum of the terms of $f$ with positive exponents, that is, $[t^{>}]f = \sum_{n \in \ZZ_{>0}} ([t^{n}]f) t^n$.

We denote by $R[t]\langle \partial_{t}\rangle$ the Ore algebra of differential operators in $t$ with polynomial coefficients.
It is a non-commutative ring, and it has a left-action on the rings above, given by $\partial_{t}(f) = \pddt{f}$.

Those definitions can be iterated to extend them to multiple variables, 
and we group together the brackets when applicable: for example $R[[x,y]]$ is the ring of formal power series in $x,y$ with coefficients in $R$, and given $f \in R[[x,y]]$, we denote by $[x^{>}y^{0}] f$ the sum of terms of positive degree in $x$ and degree $0$ in $y$ in~$f$.

\paragraph{}
For $n,k,l,h,v,u \in \ZZ$, we denote by $q_{h,v,u;k,l;n}$ the number of walks of length $n$ which:
\begin{itemize}
  \item start at $(0,0)$ and end at $(k,l)$;\par\kern-\medskipamount
  \item never leave the upper-right quadrant $\{(x,y) \in \ZZ^{2} : x \geq 0, y \geq 0\}$;
  \item visit the horizontal boundary (excluding the origin) $\{(x,y) \in \ZZ^{2} : x > 0, y = 0\}$ exactly $h$ times;
  \item visit the vertical boundary (excluding the origin) $\{(x,y) \in \ZZ^{2} : x = 0, y > 0\}$ exactly $v$ times;
  \item visit the origin $u$ times (not counting the starting point).
\end{itemize}
The associated generating function $Q(a,b,c;x,y;t)$ is defined as
\begin{equation*}
  \label{eq:1}
  Q(a,b,c;x,y;t) = \sum_{n} t^{n} \sum_{k,l} x^{k}y^{l} \sum_{h,v,u} q_{h,v,u;k,l;n} a^{h}b^{v}c^{u}.
\end{equation*}
Note that, since there are only finitely many walks of a given length, for each $n$, the two innermost sums define a polynomial.
Hence $Q(a,b,c;x,y;t)$ lives in $\QQ[a,b,c,x,y][[t]] \subset \QQ[a,b,c][[x,y,t]]$.
For shortness, we shall write $Q(x,y) := Q(a,b,c;x,y;t)$.
In particular, $Q(0,0)$ is the generating function counting interacting walks ending at $(0,0)$, $Q(x,0)$ is the generating function counting interacting walks ending on the horizontal axis and $Q(0,y)$ is the generating function counting interacting walks ending on the vertical axis.

Finally, we denote by $Q_{i,j} = Q_{i,j}(a,b,c;t) := [x^{i}y^{j}]Q(x,y) \in \QQ[a,b,c][[t]]$ the generating function counting interacting walks  ending at point $(i,j)$.
The coefficient of $t^{n}$ in $Q_{i,j}$ is a polynomial in $a,b,c$, and its coefficient for the monomial $a^{h}b^{v}c^{u}$ is exactly $q_{h,v,u;k,l;n}$.

The elements $a,b,c$ are called weights associated respectively to the horizontal boundary (excluding the origin), the vertical boundary (excluding the origin) and the origin.

\paragraph{}
Given a step set $\mathcal{S} \subseteq \{-1,0,1\}^{2} \setminus \{(0,0)\}$, the step generator $S$ is
\begin{equation*}
  \label{eq:2}
  S(x,y) = \sum_{(i,j) \in \mathcal{S}} x^{i}y^{j} \in \QQ[x,1/x,y,1/y].
\end{equation*}
We denote by
\begin{itemize}
  \item $A(x,y) = \sum_{(i,-1) \in \mathcal{S}} x^{i}y^{-1}$ the step generator for the steps going southwards;
  \item $B(x,y) = \sum_{(-1,j) \in \mathcal{S}} x^{-1}y^{j}$ the step generator for the steps going westwards;
  \item $G(x,y) = x^{-1}y^{-1}$ if $(-1,-1) \in \mathcal{S}$ and $0$ otherwise, the step generator for the steps going south-westwards.
\end{itemize}

The \emph{kernel} of the step set is $K(x,y) = 1 - tS(x,y) \in \QQ[t,x,1/x,y,1/y]$.

The kernel equation is a functional equation satisfied by the generating function counting walks restricted to the quarter plane.

\begin{thm}[{\cite[Theorem 1]{Beaton_2020}}]
  For a lattice walk restricted to the quarter-plane, starting at the origin, with weights $a$ (resp. $b$) associated with vertices on the $x$-axis excluding the origin (resp. the $y$ axis excluding the origin), and weight $c$ associated with the origin, the generating function $Q(x,y)$ satisfies the following functional equation
  \begin{multline}
    \label{eq:kernel}
    K(x,y)Q(x,y) = \frac{1}{c} + \frac{1}{a}\left( a-1-taA(x,y)\right)Q(x,0)
    + \frac{1}{b}\left( b-1-tbB(x,y) \right)Q(0,y)\\
    + \left( \frac{1}{abc}(ac + bc - ab - abc) + t G(x,y) \right)Q(0,0).
  \end{multline}
\end{thm}

\section{Main result and the power series $\Theta$}
\label{sec:kern-equat-seri}

We consider specifically the Kreweras step set $\mathcal{S} = \{(1,1),(-1,0),(0,-1)\}$.
By exhaustive enumeration, the generating function
$Q(a,b,c;x,y;t)\in  \QQ[a,b,c,x,y][[t]]$ starts 
\[ 
1+xy \, t+ \left( {x}^{2}{y}^{2}+ax+by \right) {t}^{2}+ \left( {x}^{3}{y}^{3}+ (a+1){x}^{2}y+(b+1)x{y}^{2}+ac+bc \right) {t}^{3}+\cdots.
\]  
For instance, at length 3,
the walk $(0,0) \rightarrow (1,1) \rightarrow (2,2) \rightarrow (3,3)$
 corresponds to the term $a^0 b^0 c^0 x^3 y^3 t^3 =  x^3 y^3 t^3$, as it does not touch any of the axes after leaving the origin, 
while the walk 
$(0,0) \rightarrow (1,1) \rightarrow (1,0) \rightarrow (0,0)$, corresponds
to $a^1 b^0 c^1 t^3 x^0 y^0 = ac t^3$, as after leaving the origin it touches the positive horizontal axis once, it returns to the origin once, but
does not touch the positive vertical axis.

The main result of this paper is that $Q$ is not algebraic (Theorem~\ref{prop:main-thm}).
In order to prove it, we define
\begin{equation}
  \label{eq:3}
  \Theta = [x^{>}y^{0}]\left( \frac{(x-y)(x^{2}y-1)(xy^{2}-1)}{xyK(x,y)} \right),
\end{equation}
where $K(x,y) = 1-t \left( xy+{x}^{-1}+{y}^{-1} \right)$.
We will prove that $\Theta$ is not algebraic.
The connection between $\Theta$ and $Q$ comes from the following lemma.

\begin{lem}[{\cite[Lemma 10]{Beaton_2020}}]
  There exist Laurent polynomials $\beta, \beta_{x,0}, \beta_{0,x}, \beta_{0,0}, \beta_{1,0}, \beta_{2,0}, \beta_{3,0} \in \QQ[t,1/t,x,1/x]$,
  such that
  \begin{multline}
    \label{eq:4}
    \beta + \beta_{x,0}Q(x,0) + \beta_{0,x}Q(0,x) + \beta_{0,0}Q(0,0) + \beta_{1,0}Q_{1,0} + \beta_{2,0}Q_{2,0} + \beta_{3,0}Q_{3,0} \\
    = t^{3} \frac{a-b}{c}(ab - (ab-ac-bc+abc)Q(0,0)) \Theta.
  \end{multline}
\end{lem}
\begin{proof}
  It is a straightforward transposition of~\cite[Lemma 10]{Beaton_2020}, by observing that with the notations therein, $\theta =  t^{3}\frac{a-b}{c} ab \Theta$ and
$\theta_{0,0} =  t^{3}\frac{a-b}{c}(ab-ac-bc+abc) \Theta$.
\end{proof}

\section{Transcendence of $\Theta$}
\label{sec:non-algebr-theta}

Recall that for $\alpha,\beta,\gamma\in\set Q$ 
and $-\gamma \notin \mathbb{N}$,                               
the Gaussian hypergeometric series $\pFqText{\alpha,\,\beta}{\gamma}{t}$ is defined as
\begin{equation*}
  \label{eq:11}
  \pFq{\alpha, \beta}{\gamma}{t} := \sum_{n=0}^{\infty} \frac{(\alpha)_n(\beta)_n}{(\gamma)_n} \, \frac{t^n}{n!}
  \in \QQ[[t]],
\end{equation*}
where 
$(u)_n$ denotes the Pochhammer symbol $(u)_n=u(u+1)\cdots(u+n-1)$ for $n\in\mathbb{N}$.
It satisfies the differential equation 
$
\left( {t}^{2}-t \right) y'' \left( t \right) +
\left( (\alpha+\beta+1)t - \gamma \right)  y' \left( t \right) + 
\alpha \beta \, y \left( t \right) = 0.
$ 
\begin{thm}
  \label{thm:theta-non-algebraic}
  The power series $\Theta$ defined in~\eqref{eq:3} admits the following closed form representation: 
  \begin{equation*}
    \label{eq:6}
    \Theta(t;x) = A_{1}(t;x) + A_{2}(t;x) \int_{0}^{t} A_{3}(s;x) T(s;x) ds, 
  \end{equation*}
  where
  \begin{equation*}
    \begin{aligned}
      A_{0} &= \sqrt{1-\frac{2t}{x}-(4x^3-1)\frac{t^{2}}{x^2}}, \\[5pt]
      A_{1} &= \frac{1}{6xt^{3}} - \frac{x^{3}-1}{2x^{2}t^{2}} + \frac{2-3x^{3}}{6x^{3}t}
      + \frac{tx^{3}+2t-x}{6t^{3}x^{2}} A_{0}, \\[5pt]
      A_{2} &= \frac{x^{2}(x-tx^{3}-2t)}{3t^{3}} A_{0}, \\[5pt]
      A_{3} &= \frac{1}{(tx^{3}+2t-x)^{2}(4t^{2}x^{3}-(x-t)^{2})A_{0}}, \\[5pt]
      T &= (3t-x)x \;\pFq{-1/3, -2/3}{1}{27t^{3}}
      + 4t(2tx^{3}+t-x)\, \pFq{-1/3,1/3}{2}{27t^{3}}.
    \end{aligned}
  \end{equation*}
  The power series $A_{0}, A_{1}, A_{2}$ and $A_{3}$ are algebraic, and the power series $T$ is transcendental.
  In particular, $\Theta$ is transcendental.
\end{thm}
\begin{proof}
We prove that $\Theta$ is equal to $C:=A_{1}+A_{2}\int A_{3}T$, in four steps:
  \begin{enumerate}
    \item use creative telescoping~\cite{Chyzak00,Koutschan10} to obtain a differential operator $L_{ct} \in \QQ(x,t)\langle \partial_{t}\rangle$ annihilating $\Theta$;
    \item verify that $L_{ct}$ annihilates~$C$; 
    \item find $r\in\mathbb N$ such that if $s \in\set Q(x)[[t]]$ is a power series solution of $L_{ct}$ and $s = 0 \mod{t^{r}}$, then $s=0$; 
    \item verify that $\Theta$ and $C$, both power series solution of $L_{ct}$, are equal modulo $t^{r}$ and so they are actually equal.
  \end{enumerate}

For the first step, define
  \begin{equation*}
    \label{eq:8}
    \Theta_{0}(t;x,y) = \frac{(x-y)(x^{2}y-1)(xy^{2}-1)}{xyK(x,y)} \in \QQ(x,y,t)
  \end{equation*}
  so that $\Theta = [x^{>}y^{0}]\Theta_{0}$.
  In order to bring the problem to a form suitable to creative telescoping algorithms, we encode the coefficient extractions as residues.
  Extracting the constant coefficient in~$y$ is immediate: for any $F \in \QQ[x,1/x,y,1/y][[t]]$, as by definition,
  \begin{equation*}
    \label{eq:9}
    [y^{0}]F(t;x,y) = \Res_{y=0} \left(\frac{F(t;x,y)}{y}\right) \in \QQ[x,1/x][[t]].
  \end{equation*}
  For extracting the positive part, we follow~\cite[Theorem~3]{BCHKP17}: for any $F \in \QQ[x,1/x][[t]]$, 
  \begin{equation*}
    \label{eq:10}
    [x^{>}]F(t;x) = \Res_{z=0}\left[ \frac{1}{z} F(t;z) \frac{\frac{x}{z}}{1 - \frac{x}{z}} \right] \in \QQ[x][[t]].
  \end{equation*}
  So composing the two, we get
  \begin{equation}
    \label{eq:7}
    \Theta(t;x) = \Res_{z=0}\Res_{y=0} \left[ \frac{1}{yz} \Theta_{0}(t;z,y) \frac{\frac{x}{z}}{1 - \frac{x}{z}} \right] .
  \end{equation}
  An annihilator for this can now be computed using creative telescoping, for example, using the Mathematica package \texttt{HolonomicFunctions}~\cite{Koutschan09}, with the following:
\begin{verbatim}
(* In Mathematica *)
<< "HolonomicFunctions.m"
Theta0 = (x-y)*(x^2*y-1)*(x*y^2-1)/(x*y*(1-t*(1/x+1/y+x*y)))
Theta0z = Theta0 /. x -> z
Lct = First[First[CreativeTelescoping[
        First[CreativeTelescoping[Theta0z/z/y * (x/z)/(1-x/z),
                                  Der[y], {Der[z],Der[t]}]],
        Der[z], {Der[t]}]]]
\end{verbatim}
  This yields an operator $L_{ct} \in\set Q(x,t)\langle \partial_{t}\rangle$ of order~6, which annihilates~$\Theta$.

  Checking that $L_{ct}$ annihilates $C$ is a straightforward computation with a computer algebra software.
  For instance, in Maple, the following command evaluates to 0:
\begin{verbatim}
# In Maple
with(DEtools);
simplify(eval(diffop2de(Lct, [Dt,t], y(t)), y(t) = C));
\end{verbatim}
  For the last step, we need to look at a basis of power series solutions of~$L_{ct}$.
Computer algebra software can again be used to compute (truncations of) elements in such a basis.
  For instance, this can be done with the following lines in Maple.
\begin{verbatim}
# In Maple
Order := 8; 
sols := formal_sol(Lct,[Dt,t]);

# Keep only the power series solutions
sols := select(s -> type(series(s,t=0),'taylor'), sols);
\end{verbatim}
 The output shows that the set of power series solutions is a $\QQ(x)$-vector space of dimension $2$ spanned, after a change of basis bringing the first terms to echelon form, by
  \begin{align*}
    s_{0} &= 1 + xt^{2} + t^{3} + O(t^{4}), \\
    s_{1} &= t + \frac{1-x^{3}}{x} t^{2} + \frac{1}{x^{2}}t^{3} + O(t^{4}).
  \end{align*}
  So a power series solution of $L_{ct}$ is entirely determined by its coefficients of degree $0$ and~$1$, 
  and in particular knowing a power series modulo $t^{2}$ is enough.
  
  Finally, checking that the first two coefficients of $C$ and $\Theta$ are equal is again a straightforward computation.
  For instance, again using Maple:
\begin{verbatim}
# In Maple
map(normal,series(C,t,5));
series(Theta,t,2);
\end{verbatim}
returns the same result $- x^{2} + O(t^{2})$. This allows to conclude that $\Theta = C$.

  For the second statement of the theorem, note that $A_{0}$, $A_{1}$, $A_{2}$ and $A_{3}$ are algebraic by closure properties of algebraic functions.
  The proof of the fact that $T$ is transcendental 
combines human observation and computer algebra.
  First, observe that if $T(t;x)$ was algebraic, then by closure properties so would be $T(t;3t)$, which has only one hypergeometric term, and thus 
$H(t) = \pFqText{-1/3,1/3}{2}{t}$ 
would be algebraic.
  But it is straightforward to verify that this cannot be the case, 
either by a lookup in Schwarz's classification of algebraic~${}_{2}F_{1}$'s~\cite{Schwarz1873},
or simply by observing that the minimal-order linear differential equation 
$
\left( 9 {t}^{2} - 9t \right) H'' \left( t \right) +
\left( 9t - 18 \right)  H' \left( t \right) - \, H \left( t \right) = 0
$ of $H$ has solutions which cannot be algebraic because one of them has logarithms in its local expansion at~$0$.

  Finally, it follows that $\Theta$ is also transcendental, again by closure properties: if $\Theta$ was algebraic, so would be $\int_{0}^{t} A_{3}(s;x) T(s;x) ds$, and so would be its derivative $A_{3}T$, and so would be~$T$.
\end{proof}

We now give some more explanation on the process we followed to find the closed form proved above.
The first step is to compute a small-order differential operator annihilating~$\Theta$.
It is possible, given the data of the series coefficients, to guess such an operator $L_{g}$ of order $4$ (hence smaller in order than $L_{ct}$), by using for instance the guesser~\cite{kauers09a}:
\begin{verbatim}
(* In Mathematica, continuation of the previous calculations *)
<< "Guess.m"
Theta = Expand[x*Expand[1/x CoefficientList[Series[Theta0,{t,0,70}],t]]
        /. (x^i_ /; i<0) -> 0 /. (y^i_ /; i!=0) -> 0];
Lg = GuessMinDE[Theta,TT[t]]
\end{verbatim}
The output is an operator $L_{g} \in\set Q(x,t)\langle \partial_{t}\rangle$ which is likely to annihilate $\Theta$.
  We can increase our trust in this operator by verifying that $L_{g}$ right-divides $L_{ct}$:
\begin{verbatim}
(* In Mathematica, continuation of the previous calculations *)
OreReduce[Lct,{ToOrePolynomial[Lg,TT[t]]}]
\end{verbatim}
  This line returns the right-remainder of $L_{ct}$ modulo $L_{g}$, which is $0$ as expected.

  At this point, we could also compute the quotient, and, examining its solutions similarly to what was done in the proof of Theorem~\ref{thm:theta-non-algebraic}, prove that $L_{g}$ annihilates~$\Theta$.
  But this is not necessary: the constructed closed form will (by design) be annihilated by $L_{g}$, so the fact that~$L_{g}$ is an annihilator of $\Theta$ is a consequence of the theorem.

  As a next step, we compute a closed form solution $C$ of $L_{g}$.
The starting point is to decompose $L_{g}$ as the least common left multiple (LCLM) of two operators of smaller order~\cite{Hoeij96}:
\begin{verbatim}
# In Maple
L1, L3 := op(DFactorLCLM(Lg, [Dt,t]));
\end{verbatim}
The output is a pair of two operators $L_1$ of order 1, and $L_3$ of order~3,
such that $L_{g} = \text{LCLM}(L_1, L_3)$.
Equivalently, in terms of solution spaces, this means that a basis of solutions of $L_g(y)=0$ is obtained by the union of the bases of $L_1$ and $L_3$, respectively. The operator $L_1$ admits a simple solution; this can be seen using the Maple command
\begin{verbatim}
dsolve(diffop2de(L1, [Dt,t], y(t)), y(t));                       
\end{verbatim}
which outputs 
\[  {\frac {3\,{x}^{3}}{t}}+{\frac {3\,{x}^{4}}{{t}^{2}}}-\frac{2}{t}+{\frac {3\,x}{{
t}^{2}}}-{\frac {{x}^{2}}{{t}^{3}}} . \]
It remains to treat the operator $L_3$. 
The starting point is to decompose it as the product of two operators of smaller order~\cite{Hoeij97}:
\begin{verbatim}
# In Maple
fac := DFactor(L3, [Dt,t]);
\end{verbatim}  
The output is a pair $\texttt{fac} = [L_2, S_1]$ of two operators of order 2, respectively~1,
such that $L_{3} = L_2 \, S_1$.   
Now the differential equation $L_3(z)=0$ is equivalent to $L_2(y)=0$ and $S_1(z) = y$.
Hence, it remains to solve  $L_2(y)=0$.  This can be done by using the algorithm in~\cite{KuHo13} and its Maple implementation provided by the authors\footnote{\url{https://www.math.fsu.edu/~vkunwar/hypergeomdeg3/hypergeomdeg3}}.
Using the command \texttt{hypergeomdeg3}, one gets a solution in terms of hypergeometric $_2F_1$ functions:  
\begin{verbatim}
SOL:=x/t^3/(t*x^3+2*t-x)/(4*t^2*x^3-t^2+2*t*x-x^2)*
        ((x-3*t)*x*hypergeom([-1/3, -2/3],[1],27*t^3)
            -4*t*(2*t*x^3+t-x)*hypergeom([-1/3, 1/3],[2],27*t^3)); 
\end{verbatim}
One can check that this is indeed a solution of $L_2$; indeed, the simplification command 
\begin{verbatim}
simplify(eval(diffop2de(fac[1], [Dt,t], y(t)), y(t) = SOL));
\end{verbatim} 
return~0. 
Moreover, one can show that this solution coincides (locally at $t=0$) with the unique power series solution of~$L_2$. 
Finally, the solution of $L_3(z)=0$ can be found using
\begin{verbatim}
simplify(dsolve( diffop2de(fac[2], [Dt,t], z(t) ) = SOL, z(t)));
\end{verbatim}
which yields             
\[
{\frac {t{x}^{3}+2t-x}{{t}^{3}}\sqrt { \left( 4{x}^{3}-1 \right) {t}^{2}+2tx-{x}^{2}} \left( \int \!{\frac {{\rm
SOL} \cdot {t}^{3}}{t{x}^{3}+2t-x}{\frac {1}{\sqrt { \left( 4{x}^{3}-1 \right) {t}^{2}+2tx-{x}^{2}}}}}
\,dt+c \right) },    
\]                           
where ${\rm SOL}$ is the hypergeometric expression found above and $c=c(x)$
is a constant function in~$t$, that is found by fitting initial terms of the power series expansions.
Putting pieces together yields the expression in the statement of Theorem~\ref{thm:theta-non-algebraic}.
 
Note that the method sketched above is rigorous in the sense that the closed
form solution of $L_{ct}$ found in this way is correct by construction. The
alternative correctness argument in the proof of the theorem is independent of
how the closed form was found.

It is also worth noting that providing a closed form for~$\Theta$ is somewhat more than
just proving its transcendence.  If we were not interested in the hypergeometric
expression, we could prove the transcendence of $\Theta$ directly on the level of
operators. For, the lclm decomposition quoted above translates into a decomposition
$\Theta=\Theta_1+\Theta_3$ where $\Theta_1$ is a solution of $L_1$ and $\Theta_3$ is a
solution of~$L_3$. Since $\Theta_1$ (stated above) is rational, transcendence of $\Theta$
is equivalent to transcendence of~$\Theta_3$. Now assume that $\Theta_3$ is
algebraic. Then the factorization $L_3=L_2S_1$ implies that $y:=S_1(\Theta_3)$ is
algebraic as well, and that it is a solution of~$L_2$. To conclude the argument, it
suffices to observe that $y\neq0$, that $S_1$ is irreducible, and that $S_1$ has a
logarithmic singularity.

\section{Transcendence of $Q$}
\label{sec:non-algebraicity-q}

\begin{thm}
  Assume that $a\neq b$ and $c \neq 0$.
  In particular, this is the case if $a,b,c$ are
variables in the polynomial ring $\QQ[a,b,c]$.
  Then the power series $Q(x,y)$, $Q(x,0)$ and $Q(0,y)$ are transcendental over $\QQ(a,b,c,x,y,t)$.
\end{thm}
\begin{proof}
  First note that the algebraicity of the three series is equivalent: if $Q(x,y)$ is algebraic, then so are its specializations $Q(0,y)$ and $Q(x,0)$; and conversely, if, say, $Q(0,y)$ is algebraic, then by symmetry of the step set so is $Q(x,0)$, and by the kernel equation, so is $Q(x,y)$.

  To reach a contradiction, assume that $Q(x,y)$ is algebraic.
  Then, by taking the derivative along $x$ and taking the value at $x=y=0$, the power series $Q_{1,0}$ is also algebraic.
  Repeating the same process, $Q_{2,0}$ and $Q_{3,0}$ are algebraic.
  Recall that $Q(0,0)$ is algebraic~\cite[Corollary~3]{Beaton_2020}.  
 So all in all, the left-hand side~$L$ of Equation~\eqref{eq:4} is algebraic.

  If $(a-b)\Big(ab - (ab - ac - bc + abc)Q(0,0)\Big) \neq 0$, this would imply that
  \begin{equation*}
    \label{eq:5}
    \Theta = \frac{c\,L}{(a-b)
      \Big(ab - (ab - ac - bc + abc)Q(0,0)\Big)
      t^{3}}
  \end{equation*}
  is also algebraic, which is a contradiction with Theorem~\ref{thm:theta-non-algebraic}.

  Thus, $(a-b)\big(ab - (ab - ac - bc + abc)Q(0,0)\big) = 0$.
  By assumption, $a \neq b$, so the second factor has to be zero.
Since $Q(0,0) = 1 + (a+b)c t^3 + \cdots$,
extracting coefficients of $t^0$ and $t^3$ in this second factor yields 
$abc=ac+bc$ and $0 = ab(a+b)c$.
Since $c\neq 0$, these relations imply $ab=a+b$ and $0=ab(a+b)$, thus
$ab=a+b=0$, and finally $a=b=0$, which contradicts the assumption $a\neq b$.
\end{proof}

\section{Particular cases and additional remarks}
\label{sec:part-cases-addit}

If $a=b$, as observed in~\cite[Section~5.5]{Beaton_2020}, the right-hand side of Equation~\eqref{eq:4} vanishes, and then the series $Q(x,y)$ is algebraic.
If $c=0$, then in particular $Q(0,0)=1$, and both sides of Equations~\eqref{eq:4} and~\eqref{eq:kernel} (after clearing out the denominator $c$) vanish.
We do not know if the power series $Q(x,y)$ is algebraic or even D-finite in that case.
With sample values of $a,b,x,y$ and $c=0$, we were not able to guess any algebraic, differential or recurrence relation with the first \num{10000} coefficients in $t$ of the series $Q(a,b,c;x,y;t)$.

The generating function $Q(1,1)$, which counts interacting walks regardless of their ending point, is also of interest, besides $Q(0,0)$ and $Q(x,y)$.
Experimentally, this generating function appears to be algebraic: we could guess a polynomial%
\footnote{Available online: \url{http://www.algebra.uni-linz.ac.at/research/kreweras-interacting/equations_Q11}}
$P(a,b,c;t,u) \in \mathbf{F}_{45007}[a,b,c,t,u]$, with $\mathbb{F}_{45007}$ the finite field with $45007$ elements, such that, for a large number of values of $(a,b,c) \in \mathbb{F}_{45007}^{3}$, one has $P(a,b,c;t,Q(a,b,c;1,1;t)) = 0 \bmod t^{2350}$. The polynomial $P$ has degree $92$ in $t$, degree $24$ in $u$, degree $60$ in $a$ and in $b$, and degree $24$ in $c$. 
In dense monomial form it has a size of more than 1GB.
The next step would be to \emph{lift} the result to obtain a polynomial $P \in \QQ[a,b,c,t,u]$, and to \emph{prove} that $P(a,b,c;t,Q(a,b,c;1,1;t)) = 0$.    
In principle, this is doable using (a variant of) the approach in~\cite{BoKa10}.

\acknowledgements{
We warmly thank Mark van Hoeij for his advice, and for his precious help with the package \href{{https://www.math.fsu.edu/~vkunwar/hypergeomdeg3/hypergeomdeg3}}{hypergeomdeg3}, especially in the parametric case.
}


\begin{thebibliography}{10}

\bibitem{Beaton_2019}
N.~R. Beaton, A.~L. Owczarek, and A.~Rechnitzer.
\newblock Exact solution of some quarter plane walks with interacting
  boundaries.
\newblock {\em The Electronic Journal of Combinatorics}, 26(3), sep 2019.

\bibitem{Beaton_2020}
Nicholas~R Beaton, Aleksander~L Owczarek, and Ruijie Xu.
\newblock {Quarter-plane lattice paths with interacting boundaries: the
  Kreweras and reverse Kreweras models}.
\newblock May 2019.

\bibitem{BoKa10}
A.~Bostan and M.~Kauers.
\newblock The complete generating function for {G}essel walks is algebraic.
\newblock {\em Proc. Amer. Math. Soc.}, 138(9):3063--3078, 2010.
\newblock With an appendix by Mark van Hoeij.

\bibitem{BCHKP17}
Alin Bostan, Fr\'{e}d\'{e}ric Chyzak, Mark van Hoeij, Manuel Kauers, and Lucien
  Pech.
\newblock Hypergeometric expressions for generating functions of walks with
  small steps in the quarter plane.
\newblock {\em European J. Combin.}, 61:242--275, 2017.

\bibitem{Chyzak00}
Fr\'{e}d\'{e}ric Chyzak.
\newblock An extension of {Z}eilberger's fast algorithm to general holonomic
  functions.
\newblock volume 217, pages 115--134. 2000.
\newblock Formal power series and algebraic combinatorics (Vienna, 1997).

\bibitem{Flajolet87}
Philippe Flajolet.
\newblock Analytic models and ambiguity of context-free languages.
\newblock volume~49, pages 283--309. 1987.
\newblock Twelfth international colloquium on automata, languages and
  programming (Nafplion, 1985).

\bibitem{Hoeij96}
Mark van Hoeij.
\newblock Rational solutions of the mixed differential equation and its
  application to factorization of differential operators.
\newblock In {\em Proceedings of the 1996 International Symposium on Symbolic
  and Algebraic Computation}, ISSAC '96, page 219–225, New York, NY, USA,
  1996. Association for Computing Machinery.

\bibitem{Hoeij97}
Mark van Hoeij.
\newblock Factorization of differential operators with rational functions
  coefficients.
\newblock {\em J. Symbolic Comput.}, 24(5):537--561, 1997.

\bibitem{kauers09a}
M.~Kauers.
\newblock Guessing handbook.
\newblock Technical Report 09-07, RISC-Linz, 2009.
\newblock Available at
  \url{http://www.risc.jku.at/publications/download/risc_3814/demo.nb.pdf}.

\bibitem{Koutschan09}
Christoph Koutschan.
\newblock {\em Advanced Applications of the Holonomic Systems Approach}.
\newblock PhD thesis, RISC, Johannes Kepler University, Linz, Austria, 2009.

\bibitem{Koutschan10}
Christoph Koutschan.
\newblock A fast approach to creative telescoping.
\newblock {\em Math. Comput. Sci.}, 4(2-3):259--266, 2010.

\bibitem{KuHo13}
Vijay~Jung Kunwar and Mark van Hoeij.
\newblock Second order differential equations with hypergeometric solutions of
  degree three.
\newblock In {\em I{SSAC} 2013---{P}roceedings of the 38th {I}nternational
  {S}ymposium on {S}ymbolic and {A}lgebraic {C}omputation}, pages 235--242.
  ACM, New York, 2013.

\bibitem{Schwarz1873}
H.~A. Schwarz.
\newblock Ueber diejenigen {F}\"{a}lle, in welchen die {G}aussische
  hyper\-geo\-me\-tri\-sche {R}eihe eine algebraische {F}unction ihres vierten
  {E}lementes darstellt.
\newblock {\em J. Reine Angew. Math.}, 75:292--335, 1873.

\bibitem{Singer14}
M.~F. Singer.
\newblock Private communication, 2014.

\bibitem{Singer80}
Michael~F. Singer.
\newblock Algebraic solutions of {$n$}th order linear differential equations.
\newblock In {\em Proceedings of the {Q}ueen's {N}umber {T}heory {C}onference,
  1979 ({K}ingston, {O}nt., 1979)}, volume~54 of {\em Queen's Papers in Pure
  and Appl. Math.}, pages 379--420. Queen's Univ., Kingston, Ont., 1980.

\bibitem{Stanley80}
R.~P. Stanley.
\newblock Differentiably finite power series.
\newblock {\em European J. Combin.}, 1(2):175--188, 1980.


\end{thebibliography}
\end{document}